\documentclass{amsart}
\usepackage{amsmath,amssymb,hyperref,mathrsfs,color}
\usepackage{paralist}

\makeatletter
\def\setliststart#1{\setcounter{\@listctr}{#1}%
  \addtocounter{\@listctr}{-1}}
\makeatother

\usepackage{tikz}
\usetikzlibrary{intersections}
\usepackage{calc}
 \newtheorem{The}{Theorem}[section]
 
 \newtheorem{Lem}[The]{Lemma}
 \newtheorem{Pro}[The]{Proposition}
 \theoremstyle{definition}
 \newtheorem{defn}[The]{Definition}
 \theoremstyle{remark}
 \newtheorem{Rem}[The]{Remark}
 
 \numberwithin{equation}{section}

\newcommand{\R}{\mathbb{R}}

\newcommand{\N}{\mathbb{N}}

\title{Lasry-Lions, Lax-Oleinik and Generalized characteristics}
\author{Cui Chen \and Wei Cheng}
\address{Department of Mathematics, Nanjing University,
Nanjing 210093, China}
\email{chenc@ujs.edu.cn}
\address{Department of Mathematics, Nanjing University,
Nanjing 210093, China}
\email{chengwei@nju.edu.cn}
\date{\today}
\begin{document}
\maketitle

\begin{abstract}
	In the recent works \cite{Cannarsa-Chen-Cheng} and \cite{Cannarsa-Cheng3}, an intrinsic approach of the propagation of singularities along the generalized characteristics was obtained, even in global case, by a procedure of sup-convolution with the kernel the fundamental solutions of the associated Hamilton-Jacobi equations. 
		
	In the present paper, we exploit the relations among Lasry-Lions regularization, Lax-Oleinik operators (or inf/sup-convolution) and generalized characteristics, which are discussed in the context of the variational setting of Tonelli Hamiltonian dynamics, such as Mather theory and weak KAM theory.
\end{abstract}

\section{Introduction}
Suppose $H=H(x,p):\R^n\times\R^n\to\R$ is a $C^2$ Tonelli Hamiltonian, i.e., $H$ is convex in $p$ with superlinear growth condition. Let $u:\R^n\to\R$ be a (global) viscosity solution of the Hamiltonia-Jacobi equation
\begin{equation}\label{eq:intro_HJ_1}
	H(x,Du(x))=0,\quad x\in\R^n.
\end{equation}
Such a solution $u$ is locally semiconcave (with linear modulus) on $\R^n$. We denote by $D^+u(x)$ the superdifferential of $u$ at $x$ (see, for instance, \cite{Cannarsa-Sinestrari}), which is a compact convex set in $\R^n$, and we call $x\in\R^n$ a singular point of $u$ if $D^+u(x)$ is not a singleton. Certain ``singular dynamics'' was interpreted in \cite{Albano-Cannarsa} by a Hamiltonian inclusion
$$
\dot{\mathbf{x}}(s)\in\text{co}\, H_p(\mathbf{x}(s),D^+u(\mathbf{x}(s))),\quad a.e.\,s\in[0,\tau],
$$
and such a Lipschictz arc $\mathbf{x}$ is called a {\em generalized characteristic}. If $x_0$ is a singular point of $u$ and 
\begin{equation}\label{eq:intro_non_critical}
	0\not\in\text{co}\,H_p(x_0,D^+u(x_0)),
\end{equation}
then the associated generalized characteristic $\mathbf{x}(t)$, $t\in[0,\tau]$, is composed of singular points of $u$. In the recent works \cite{Cannarsa-Chen-Cheng} and \cite{Cannarsa-Cheng3}, the propagation of singularities along generalized characteristics in \cite{Cannarsa-Cheng3} has been explained by an intrinsic variational approach (see, also, \cite{Albano-Cannarsa} \cite{acns} \cite{Ambrosio-Cannarsa-Soner} \cite{Cannarsa-Cheng} \cite{Cannarsa-Sinestrari} \cite{Cannarsa-Yu} \cite{Yu} for the approach from Control theory or PDE), which is motivated by Mather theory and weak KAM theory. 

Let us recall the aforementioned results in \cite{Cannarsa-Chen-Cheng} and \cite{Cannarsa-Cheng3} at first. Let $u\in C(\R^n)$, for any $t>0$, $\breve{T}_t$, the Lax-Oleinik operator of positive type, is defined as 
\begin{equation}\label{eq:intro_sup_convolution}
	\breve{T}_tu(x):=\sup_{y\in\R^n}\{u(y)-A_t(x,y)\}:=\sup_{y\in\R^n}\psi_t(y),\quad x\in\R^n,
\end{equation}
where 
$$
A_{t}(x,y)=\inf_{\gamma\in\Gamma^t_{x,y}}\int^{t}_{0}L(\gamma(s),\dot{\gamma}(s))ds
$$
with
$$
\Gamma^t_{x,y}=\{{\gamma\in W^{1,1}([0,t],\R^n): \gamma(0)=x,\gamma(t)=y}\}.
$$
Here $L$ is an arbitrary Tonelli Lagrangian on $\R^n$ with $H$ its Fenchel-Legendre dual, and it is well known that \eqref{eq:intro_sup_convolution} is also called sup-convolution or Lax-Oleinik opterators in the literature. In \cite{Cannarsa-Chen-Cheng} and \cite{Cannarsa-Cheng3}, the authors shown that the maximizers in such a procedure of sup-convolution give exactly a local or global generalized characteristic starting from a singular points of $u$ under suitable conditions. 

In the present paper, we will explain the connection between generalized characteristics and the well-known Lasry-Lions regularization at first. Throughout this paper, we suppose that $L$ satisfies condition (L1) and (L2) (see Section 2).

Let $M$ be a $C^2$ closed manifold, $t>0$ and $u:M\to\R$ is any semiconcave function, the following properties are already known (see, for instance, \cite{Bernard2007}, \cite{Fathi-book})
\begin{enumerate}[(P1)]
  \item $\breve{T}_tu$ belongs to class $C^{1,1}$ for $0<t\leqslant t_0$ with $t_0$ is a constant dependent on the constant of semiconcavity of $u$.
  \item $\breve{T}_tu$ is decreasing on $(0,+\infty)$ if $u$ is a viscosity subsolution of Hamilton-Jacobi equation
$$
H(x,Du(x))=\alpha(0),\quad x\in\ M,
$$
where $\alpha(\cdot)$ is Mather's $\alpha$-function. Moreover, $\breve{T}_tu$ tends to $u$ uniformly as $t\to 0^+$.
\end{enumerate}

In this paper, we also have

\begin{The}\label{LL-Regularization}
Suppose $u:\R^n\to\R$ is a semiconcave function. Then there exists $0<t_0\ll 1$ such that 
\begin{enumerate}[(P1)]\setliststart{3}
  \item Let $x_0\in\R^n$ and $L(x_0,0)\leqslant 0$, then $\breve{T}_tu(x_0)$ is increasing on $(0,+\infty)$ and $\lim_{t\to 0^+}\breve{T}_tu(x_0)=u(x_0)$. Consequently, if 
  $$
  L(x,0)\leqslant 0,\quad \forall x\in\R^n,
  $$
  then $\breve{T}_tu$ tends to $u$ uniformly on any compact subset as $t\to 0^+$.
  \item Let $x\in\R^n$, suppose that the function $\psi_t$ defined in \eqref{eq:intro_sup_convolution} attains the maximizer $y_t$ in $B(x,R(x,t))$ for all $t>0$ where $R(x,t)>0$ is defined in \eqref{eq:key}, then $\lim_{t\to0^+}D\breve{T}_tu(x)=p_0$, where $p_0$ is the unique element with minimal energy:
  \begin{equation}\label{eq:ME}
	H(x_0,p)\geqslant H(x_0,p_0),\quad \forall p\in D^+u(x_0).\tag{ME}
  \end{equation}
\end{enumerate}
\end{The}

It is worth noting that the minimal energy condition in \eqref{eq:ME} is the same as the inititial condition on the velocity of the generalized characteristic obtained by the intrinsic approach in \cite{Cannarsa-Chen-Cheng} and \cite{Cannarsa-Cheng3}, see also Proposition \ref{sing_arc_gen_char}.

In the rest part of this paper, we try to exploit the nature of the singularities of $u$ by the procedure of inf-convolution. As pointed out in \cite{Cannarsa-Chen-Cheng}, the inf-convolution defined by
\begin{equation}\label{eq:intro_inf_convolution}
	T_tu(x):=\inf_{y\in\R^n}\{u(y)+A_t(y,x)\},\quad x\in\R^n,
\end{equation}
is not the dual procedure of sup-convolution. But, it is still meaningful for study the critical points of the local barrier function
$$
\phi_t(x)=u(x)+A_t(x,x_0)
$$
with respect to any fixed point $x_0$. Recall that a point $x\in\R^n$ is a {\em critical point} of a locally semiconcave function $u$ if $0\in D^+u(x)$. Comparing to the local barrier function
$$
\psi_t(x)=u(x)-A_t(x_0,x),
$$
the function $\psi_t$ only admits a unique critical point (maximizer) for small time $t>0$ since the convexity properties of the fundamental solutions $A_t(x_0,x)$ (see, Appendix A).

Along this line, given a singular point $x$ of $u$, using a nonsmooth critical point theorem by Shi (\cite{Shi}), we obtain a critical point of the local semiconcave function $\phi_t$, which is not a global minimizer of $\phi_t$ determined by classical characteristic passing to $x$.

\begin{The}\label{intro:thm2}
	Let $u$ be a Lipschitz viscosity solution of \eqref{eq:intro_HJ_1}, $t>0$, and let $x\in\R^n$ be a singular point of $u$. Suppose there exist finite many elements in $D^*u(x)$, say $D^*u(x)=\{p_1,\ldots,p_k\}$ with $k\geqslant2$, then there exist critical points $\{x^{ij}_t\}$ of $\phi_t$ (not global minimizers) such that, for $1\leqslant i,j\leqslant k$, $i\not= j$, each critical point $x_t=x^{ij}_t$ has the following dichotomy: 
\begin{enumerate}[(a)]
  \item $x_t$ is a differentiable point of $\phi_t$ and there exists a local minimal curve connecting $x_t$ and $x$. More precisely, there exists a $C^1$ curve $\gamma:(-\infty,t]\to\R^n$ such that $\gamma(0)=x_t$, $\gamma(t)=x$ and the restriction of $\gamma$ on $(-\infty,0]$ is a $(u,L,0)$-calibrated curve, but $\gamma$ is not a $(u,L,0)$-calibrated curve on $(-\infty,t]$;
  \item $x_t$ is a singular point of $u$.
\end{enumerate}
\end{The}

From the theorem above, the location of singularities affords possible information to construct ``local'' minimal orbits for Tonelli Lagrangian systems, which is totally unknown before. In the previous works of variational approach of Hamiltonian dynamical instability problem like Arnold diffusion (see, for instance, \cite{Bernard2008}, \cite{Cheng-Yan2004}, \cite{Cheng-Yan2009}, \cite{Cheng2013}, \cite{Cheng2015-1}, \cite{Cheng2015-2} and \cite{Cheng-Xue}), the diffusion orbits shadow the variational minimizers which are not local ones.

The paper is organized as follows: In section 2, we review some basic properties of viscosity solution of Hamilton-Jacobi equations. In section 3, we discuss the relation of the generalized characteristics give by the procedure of sup-convolution and Lasry-Lions regularization, then, we also discuss what happens for the procedure of inf-convolution.

\noindent{\bf Acknowledgments} This work was partially supported by the Natural Scientific Foundation of China (Grant No. 11271182 and No. 11471238), the  National Basic Research Program of China (Grant No. 2013CB834100). The authors are grateful to Liang Jin for helpful discussions on the results of this paper.

\section{Viscosity solutions and weak KAM theory}
A $C^2$ function $L:\R^n\times\R^n\to\R$ is said to be a {\em Tonelli Lagrangian} if the following assumptions are satisfied.
\begin{enumerate}[(L1)]
  \item The Hessian $\frac{\partial^2 L}{\partial v^2}(x,v)$ is positive definite for all $(x,v)\in\R^n\times\R^n$.
  \item there exists a non-decreasing function $\theta:[0,+\infty)\to[0,+\infty)$, $\theta(r)/r\to+\infty$ as $t\to +\infty$, $c_0\geqslant0$ and ,$c_1=c_1(x,R)\geqslant 0$ such that
$$
L(x,v)\geqslant\theta(|v|)-c_0,\quad(x,v)\in\R^n\times \R^n,
$$
and 
$$
|L_x(y,v)|+|L_v(y,v)|\leqslant c_1(x,R)\theta(|v|),\quad(y,v)\in\bar{B}(x,R)\times \R^n.
$$

\end{enumerate}

Let $H:\R^n\times\R^n\to\R$ be the associated {\em Tonelli Hamiltonian}, i.e.,
$$
H(x,p)=\sup_{v\in\R^n}\{\langle p,v\rangle-L(x,v)\}.
$$

Throughout this paper, we suppose $L$ is a $C^2$ Tonelli Lagrangian with conditions (L1) and (L2).

\subsection{semiconcave functions}
Let $\Omega\subset\R^n$ be a convex open set, a function $u:\Omega\rightarrow\R$ is {\em semiconcave} (with linear modulus) if there exists a constant $C>0$ such that
\begin{equation}\label{eq:SCC}
\lambda u(x)+(1-\lambda)u(y)-u(\lambda x+(1-\lambda)y)\leqslant\frac C2\lambda(1-\lambda)|x-y|^2
\end{equation}
for any $x,y\in\Omega$ and $\lambda\in[0,1]$.  Any constant $C$ that satisfies the above inequality  is called a {\em semiconcavity constant} for $u$ in $\Omega$. A function $u:\Omega\rightarrow\R$ is said to be {\em semiconvex} (with linear modulus) if $-u$ is semiconcave. A function $u:\Omega\rightarrow\R$ is said to be {\em locally semiconcave} (resp. {\em locally semiconvex}) if for each $x\in\Omega$, there exists an open ball $B(x,r)\subset\Omega$ such that $u$ is a semiconcave (resp. semiconvex) function on $B(x,r)$.

\begin{defn}
Let $u:\Omega\subset\R^n\to\R$ be a continuous function. We recall that, for any $x\in\Omega$, the closed convex sets
\begin{align*}
D^-u(x)&=\left\{p\in\R^n:\liminf_{y\to x}\frac{u(y)-u(x)-\langle p,y-x\rangle}{|y-x|}\geqslant 0\right\},\\
D^+u(x)&=\left\{p\in\R^n:\limsup_{y\to x}\frac{u(y)-u(x)-\langle p,y-x\rangle}{|y-x|}\leqslant 0\right\}.
\end{align*}
are called the {\em (Dini) subdifferential} and {\em superdifferential} of $u$ at $x$, respectively.
\end{defn}

\begin{defn}
Let $u:\Omega\to\R$ be locally Lipschitz. We recall that a vector $p\in\R^n$ is called a {\em limiting differential} of $u$ at $x$ if there exists a sequence $\{x_n\}\subset\Omega\setminus\{x\}$ such that $u$ is differentiable at $x_k$ for each $k\in\N$, and
$$
\lim_{k\to\infty}x_k=x\quad\text{and}\quad \lim_{k\to\infty}Du(x_k)=p.
$$
The set of all limiting differentials of $u$ at $x$ is denoted by $D^{\ast}u(x)$.
\end{defn}

The fundamental properties of the superdifferential of a semiconcave function are listed in the following proposition. The monograph \cite{Cannarsa-Sinestrari} is a good reference for the topic of semiconcave functions and beyond.

\begin{Pro}\label{basic_facts_of_superdifferential}
Let $u:\Omega\subset\R^n\to\R$ be a semiconcave function and let $x\in\Omega$. Then the following properties hold.
\begin{enumerate}[\rm {(}a{)}]
  \item $D^+u(x)$ is a nonempty compact convex set in $\R^n$ and $D^{\ast}u(x)\subset\partial D^+u(x)$, where  $\partial D^+u(x)$ denotes the topological boundary of $D^+u(x)$.
  \item The set-valued function $x\rightsquigarrow D^+u(x)$ is upper semicontinuous.
  \item If $D^+u(x)$ is a singleton, then $u$ is differentiable at $x$. Moreover, if $D^+u(x)$ is a singleton for every point in $\Omega$, then $u\in C^1(\Omega)$.
  \item $D^+u(x)=\mathrm{co}\, D^{\ast}u(x)$.
  \item $D^{\ast}u(x)=\big\{\lim_{i\to\infty}p_i~:~ p_i\in D^+u(x_i),\; x_i\to x,\;\mathrm{diam}\,(D^+u(x_i))\to 0\big\}$.
\end{enumerate}
\end{Pro}

From proximal analysis point of view, the following result characterizes the semiconcavity of a continuous function and its superdifferential.

\begin{Pro}
\label{criterion-Du_semiconcave2}
Let $u:\Omega\to\R$ be a continuous function. If there exists a constant $C>0$ such that, for any $x\in\Omega$, there exists $p\in\R^n$ such that
\begin{equation}\label{criterion_for_lin_semiconcave}
u(y)\leqslant u(x)+\langle p,y-x\rangle+\frac C2|y-x|^2,\quad \forall y\in\Omega,
\end{equation}
then $u$ is semiconcave with constant $C$ and $p\in D^+u(x)$.
Conversely,
if $u$ is semiconcave  in $\Omega$ with constant $C$, then \eqref{criterion_for_lin_semiconcave} holds for any $x\in\Omega$ and $p\in D^+u(x)$.
\end{Pro}

Finally, we introduce the concept of singularity of a semiconcave function. A point $x\in\Omega$ is called a {\em singular point} of $u$ if $D^+u(x)$ is not a singleton. The set of all singular points of $u$, also called the {\em singular set} of $u$, is denoted by $\Sigma_u$.

\subsection{Fundamental solutions and viscosity solutions}
Given $x, y\in\R^n$, we define
$$
\Gamma^t_{x,y}=\{{\gamma\in W^{1,1}([0,t],\R^n): \gamma(0)=x,\gamma(t)=y}\}
$$
Let $t>0$, we denote
\begin{equation}\label{fundamental_solution}
A_{t}(x,y)=\inf_{\gamma\in\Gamma^t_{x,y}}\int^{t}_{0}L(\gamma(s),\dot{\gamma}(s))ds.
\end{equation}
It is well know that the infimum can be achieved by $C^2$ curves. In the literature of PDEs, $A_t(x,y)$ is called a {\em fundamental solution} of \eqref{HJ1}, see, for instance, \cite{McEneaney-Dower}.

Throughout this section, we suppose the $C^2$ Tonelli Lagrangian $L$ satisfies condition (L1)-(L2). We discussed the associated Nagumo type conditions and the essential regularity results of the fundamental solutions in Appendix \ref{app_sec_regularity}. For the main regularity results we will use, see, Proposition \ref{convexity_A_t} and Proposition \ref{C11_A_t}.

Suppose $H$ is a Tonelli Hamiltonian, throughout this paper we will be concerned with the Hamilton-Jacobi equation
\begin{equation}\label{HJ1}
H\big(x,Du(x)\big)=0,\qquad x\in\R^n.
\end{equation}

We recall that a continuous function $u$ is called a {\em viscosity subsolution} of equation
\eqref{HJ1} if, for any $x\in\R^n$,
\begin{align}\label{viscosity subsolution}
H(x,p)\leqslant0,\quad\forall p\in D^+u(x)\,.
\end{align}
Similarly, $u$ is a {\em viscosity supersolution} of equation \eqref{HJ1} if, for any $x\in\R^n$,
\begin{align}\label{viscosity supersolution}
H(x,p)\geqslant0,\quad\forall p\in D^-u(x)\,.
\end{align}
Finally, $u$ is called a {\em viscosity solution} of equation \eqref{HJ1}, if it is both a viscosity subsolution and a supersolution.

\begin{Pro}
Any viscosity solution of the Hamilton-Jacobi equation \eqref{HJ1} is locally semiconcave with linear modulus.
\end{Pro}

\begin{Pro}\label{Ext_and_reachable}
$\mathrm{Ext}\,D^+u(x)=D^{\ast}u(x)$\footnote{For any convex closed subset of $\R^n$, we denote by $\mathrm{Ext}\,C$ the set of extremal points of $C$.} for any viscosity solution $u$ of \eqref{HJ1} and any $x\in\R^n$.
\end{Pro}

\begin{Pro}\label{reachable_grad_and_backward}
Let $x\in \R^n$ and $u:\R^n\to\R$ be a viscosity solution of the Hamilton-Jacobi equation \eqref{HJ1}. Then $p\in D^{\ast}u(x)$ if and only if there exists a unique $C^2$ curve $\gamma:(-\infty,0]\to \R^n$ with $\gamma(0)=x$ which is which is a $(u,L,0)$-calibrated curve\footnote{For the concept of dominated functions and calibrated curves, see, for instance, \cite{Fathi-book}}, and $p=L_v(x,\dot{\gamma}(0))$.
\end{Pro}

\subsection{Generalized characteristic}
The construction of the singular set or cut loci of viscosity solutions is a very important and hard problem in many fields such as Riemannian geometry, optimal control, classical mechanics, etc.. It is known that the study of propagation of singularities can go back to \cite{alca99} for general semiconcave functions by the method from nonsmooth analysis. Some dynamical nature of the singularity was found by the concept of generalized characteristic.

\begin{defn}
A Lipschitz arc $\mathbf{x}:[0,\tau]\to\R^n$ is said to be a {\em generalized characteristic} of the Hamilton-Jacobi equation \eqref{HJ1} if $\mathbf{x}$ satisfies the differential inclusion
\begin{equation}\label{generalized_characteristics}
\dot{\mathbf{x}}(s)\in\mathrm{co}\, H_p\big(\mathbf{x}(s),D^+u(\mathbf{x}(s))\big),\quad \text{a.e.}\;s\in[0,\tau]\,.
\end{equation}
\end{defn}

A basic criterion for the propagation of singularities along generalized characteristic was given in \cite{Albano-Cannarsa} (see \cite{Cannarsa-Yu,Yu} for an improved version and simplified proof).

\begin{Pro}[\cite{Albano-Cannarsa}]\label{criterion_on_gen_char}
Let $u$ be a viscosity solution of Hamilton-Jacobi equation \eqref{HJ1} and let $x_0\in\R^n$. Then there exists a generalized characteristic  $\mathbf{x}:[0,\tau]\to\R^n$ with initial point $\mathbf{x}(0)=x_0$. Moreover, if
$x_0\in\Sigma_u$, then $\mathbf{x}(s)\in\Sigma_u$ for all $s\in [0,\tau]$. Furthermore, if
\begin{equation}\label{condition_for_propagation_singularities}
0\not\in\mathrm{co} \,H_p(x_0,D^+u(x_0))\,,
\end{equation}
 then  $\mathbf{x}(\cdot)$ is injective for every $s\in[0,\tau]$.
\end{Pro}

\section{Procedure of sup-convolution and generalized characteristics}
Let $H$ be a Tonelli Hamiltonian on $\R^n$. Recall the Lax-Oleinik operators $T_t$ and $\breve{T}_t$, i.e., for any $u\in C(\R^n)$,
\begin{gather}
\breve{T}_tu(x):=\sup_{y\in\R^n}\{u(y)-A_t(x,y))\},\label{L-L regularity_sup}\\
T_tu(x):=\inf_{y\in\R^n}\{u(y)+A_t(y,x))\}.\label{L-L regularity_inf}
\end{gather}
When taking $H(p)=|p|^2/2$ and the kernel $A_t(x,y)=\frac 1{2t}|x-y|^2$, the two operators above are closely linked to the so-called Lasry-Lions regularization procedure (\cite{Lasry-Lions}) which is written in the form of {\em sup-convolution} and {\em inf-convolution}, respectively. This type of regularization is also called Moreau-Yosida regularization in convex analysis. A more detailed formulation can be found in \cite{Attouch} with respect to the aforementioned quadratic kernel.

\subsection{Procedure of sup-convolution and generalized characteristics}\label{procedure_of_sup_convolution}
Recently, in \cite{Cannarsa-Chen-Cheng} and \cite{Cannarsa-Cheng3}, the authors studied the intrinsic relation of propagation of singularities along the generalized characteristics and the following procedure of sup-convolution.

Fix $x\in\R^n$, $0<t\leqslant t_0\ll 1$, then there exists $R(x,t)>0$ such that, the function
\begin{equation}\label{eq:psi_t}
	\psi_t(y):=u(y)-A_t(x,y),\quad y\in\bar{B}(x,R(x,t)),
\end{equation}
has a unique maximizer for each $t\in(0,t_0]$, where $A_t(x,y)$ is a fundamental solution with respect to the associated Tonelli Lagrangian $L$. 

Suppose that $u(\cdot)$ is semiconcave while $A_t(x,\cdot)$ is locally semiconcave (Proposition \ref{semiconcavity_A_t}) and convex when $t\in(0,t_0]$ (Proposition \ref{convexity_A_t}), say $C_1>0$ (resp. $-C_2(t)$) is the semiconcavity (resp. semiconvexity) constant of $u(\cdot)$ (resp. $A_t(x,\cdot)$). Note that, by Proposition \ref{C11_A_t}, the constant $C_2(t)=\frac Ct$, thus $\psi_t(\cdot)$ is strictly concave in $\bar{B}(x,R(x,t))$ and consequently we have a unique maximizer for each $t\in(0,t_0]$, which is also a unique critical point of $\psi_t$.

Let us define the arc $\mathbf{y}:[0,t_0]\to\R^n$ by
\begin{equation}\label{maximizer_arc_1}
\mathbf{y}(t)=\begin{cases}
x,&t=0,\\
y_t,&t\in(0,t_0].
\end{cases}
\end{equation}
If $\xi_t:[0,t]\to\R^n$ is the unique minimizer in the definition of $A_t(x,y)$, we define
\begin{equation}\label{maximizer_arc_2}
p_t(s):=L_v(\xi_t(s),\dot{\xi}_t(s)),\quad s\in [0,t],
\end{equation}
the associated dual arc with respect to $\xi_t(s)$.

\begin{Pro}[\cite{Cannarsa-Chen-Cheng}]\label{sing_arc_gen_char}
Let $u$ be a locally semiconcave function and $x\in\Sigma_u$, the singular set of $u$. There exists $t_1\leqslant t_0$ such that the arc $\mathbf{y}:[0,t_1]\to\R^n$ defined in \eqref{maximizer_arc_1} is a generalized characteristic composed of singular points of $u$, i.e., $\mathbf{y}:[0,t_1]\to\R^n$ is Lipschitz continuous, $\mathbf{y}(t)\in\Sigma_{u}$ for all $t\in[0,t_1]$, and satisfies
\begin{equation}\label{eq:sing_gen_char}
\dot{\mathbf{y}}(\tau)\in\mbox{\rm co}\, H_p(\mathbf{y}(\tau),D^+u(\mathbf{y}(\tau))),\quad\text{a.e.}\ \tau\in[0,t_1].
\end{equation}
Moreover, 
\begin{equation}\label{eq:strong_gen_char_1}
\dot{\mathbf{y}}^+(0)=H_p(x,p_0),
\end{equation}
where $p_0$ is the unique element of minimal energy:
$$
H(x,p)\geqslant H(x,p_0),\quad \forall p\in D^+u(x).
$$
\end{Pro}

\subsection{Lasry-Lions regularization}
In this section, we will explain the connection between Lasry-Lions regularization (\cite{Lasry-Lions}) and generalized characteristics first found in \cite{Albano-Cannarsa}. We only concentrate to the case of sup-convolution $\breve{T}_tu$ with $u$ a locally semiconcave function.

For $t>0$, recalling that
\begin{gather}
\breve{T}_tu(x):=\sup_{y\in\R^n}\{u(y)-A_t(x,y))\},\label{L-L regularity}
\end{gather}
where $u:\R^n\to\R$ is any locally semiconcave function, and $A_t(x,y)$ is the fundamental solution with respect to any Tonelli Lagrangian $L$.


\begin{The}\label{LL-Regularization}
Suppose $u:\R^n\to\R$ be a semiconcave function with constant $C$. Then there exists $0<t_0\ll 1$ such that if $\breve{T}_tu$ is defined as in \eqref{L-L regularity}, we have
\begin{enumerate}[(P1)]\setliststart{3}
  \item Fix $x\in\R^n$, then $\breve{T}_tu(x)$ is increasing on $(0,+\infty)$ and $\lim_{t\to 0^+}\breve{T}_tu(x)=u(x)$ if $L(x,0)\leqslant 0$. Consequently, if 
  $$
  L(x,0)\leqslant 0,\quad \forall x\in\R^n,
  $$
  then $\breve{T}_tu$ tends to $u$ uniformly on any compact subset as $t\to 0^+$.
  \item Let $x\in\R^n$, suppose that the function $\psi_t$ defined in \eqref{eq:psi_t} attains the maximizer $y_t$ in $B(x,R(x,t))$ for all $t>0$, then $\lim_{t\to0^+}D\breve{T}_tu(x)=p_0$, where $p_0$ is the unique element with minimal energy, i.e.,
\begin{equation}
H(x,p_0)=\min_{p\in D^+u(x)}H(x,p).
\end{equation}
  \item In particular, when $L$ has the form$$
L(x,v)=\frac 12\langle Av,v\rangle,\quad x\in\R^n, v\in\R^n,
$$
where $A$ is an $n\times n$ symmetric and positive definite matrix. If $t\leqslant \kappa C^{-1}$, then, the functions $u$ and $\breve{T}_tu$ have the same critical points and critical values where $\kappa>0$ is the smallest eigenvalue of $A$.
\end{enumerate}
\end{The}

\begin{Rem}
The properties (P1) and (P2) (see the introduction) is already known (in the case of compact manifolds), see, for instance, \cite{Bernard2007} or Fathi's book \cite{Fathi-book}. Since this is a local result, it is not hard to generalize to the manifolds using local charts. We collect the known results here just for the comparison interests like (P2) (in the introduction) and (P3). The property (P5) is a slight generalization of a known result (\cite{Attouch}).
\end{Rem}

\begin{Rem}
	It is worth noting that the assumption in (P4) that the function $\psi_t$ defined in \eqref{eq:psi_t} attains the maximizer $y_t$ in $B(x,R(x,t))$ for all $t>0$, is not easy to be checked in general. Fortunately, if we consider a certain type of nearly integrable systems or mechanical systems, this condition holds. The readers can refer to \cite{Cannarsa-Cheng3}.
\end{Rem}

\begin{proof}
Let $x,y\in\R^n$ and $t>0$, for any $0<s<t$, it is easily checked that
$$
A_t(x,y)\leqslant A_s(x,y)+A_{t-s}(x,x).
$$
Taking the constant curve $\gamma(\tau)\equiv x$, $\tau\in[0,t-s]$, we have
$$
A_{t-s}(x,x)\leqslant\int^{t-s}_0L(\gamma(\tau),\dot{\gamma}(\tau))\ d\tau=(t-s)L(x,0).
$$
Therefore, for any fixed $x\in\R^n$, we have $A_t(x,\cdot)\leqslant A_s(x,\cdot)$ since $L(x,0)\leqslant 0$, and thus, $\psi_t(y)\geqslant\psi_s(y)$ for all $y\in\R^n$. This leads to the conclusion that $\breve{T}_su(x)\leqslant \breve{T}_tu(x)$ if $0<s<t$. The uniform convergence result is a direct consequence of Dini's Lemma on monotone sequence of continuous functions. This completes the proof of (P3).

Now, we turn to the proof of (P4). Fix $x\in\R^n$ and $t\in(0,t_0]$. Adopting the same terminology as before, since $\psi_t(\cdot)$ attains the maximum at $y=y_t\in B(x,R(x,t))$ and $\xi_t\in\Gamma^t_{x,y_t}$ is the minimal curve in the definition of $A_t(x,y_t)$, we have
$$
L_v(\xi_t(t),\dot{\xi}_t(t))=D_yA_t(x,y_t)\in D^+u(y_t),
$$
since the results in Proposition \ref{C11_A_t} and $0\in D^+\psi_t(y_t)$. Moreover, we have that the family $\{\dot{\xi}_t\}_{t\in(0,t_0]}$ is equi-Lipschitz, by Proposition \ref{equi_Lip}.

Therefore, we have
\begin{align*}
\left|\frac{\xi_{t}(t)-x}{t}-\dot{\xi}_{t}(0)\right|
&\leqslant\frac 1{t}\int^{t}_{0}|\dot{\xi}_{t}(s)-\dot{\xi}_{t}(0)|\ ds\\
&\leqslant\frac C{t}\int^{t}_{0}s\ ds=\frac C2t.
\end{align*}
Thus, we obtain
$$
v_0=\lim_{t\to0^+}v_t=\lim_{t\to0^+}\dot{\xi}_t(0)=\lim_{t\to0^+}\dot{\xi}_t(t),
$$
where $v_t=(y_t-x)/t$. Since $u$ is a locally semiconcave function, thus, by the monotone property of semiconcave functions (see, e.g., \cite{Cannarsa-Sinestrari}), we have
\begin{equation}\label{basic_inequalty}
\langle p-L_v(\xi_t(t),\dot{\xi}_t(t)), v_t\rangle+tC|v_t|^2\geqslant 0,\quad \forall p\in D^+u(x).
\end{equation}
Taking limit in \eqref{basic_inequalty}, then
\begin{equation}\label{eq:limit_1}
	\langle p,v_0\rangle\geqslant\langle L_v(x,v_0),v_0\rangle,\quad \forall p\in D^+u(x).
\end{equation}
In other words,
\begin{equation}\label{optimization_problem}
H(x,p)\geqslant\langle L_v(x,v_0),v_0\rangle-L(x,v_0)=H(x,p_0),\quad \forall p\in D^+u(x),
\end{equation}
where $p_0=L_v(x,v_0)\in D^+u(x)$, by the upper semicontinuity of the set valued function $x\leadsto D^+u(x)$, is the unique element solve the associated optimization problem \eqref{optimization_problem}, and $\lim_{t\to0^+}D\breve{T}_tu(x)=\lim_{t\to0^+}L_v(\xi_t(0),\dot{\xi}_t(0))=p_0$. This completes the proof of (P4).

For the proof of (P5), note that, in our case, the minimal curve $\xi_t(s)=\frac{y_t-x}t\cdot s$, thus, by \eqref{basic_inequalty}, we have
$$
\langle p-Av_t, v_t\rangle+tC|v_t|^2\geqslant 0,\quad \forall p\in D^+u(x).
$$
If $0\in D^+u(x)$, take $p=0$ in the inequality above, then it follows
$$
\langle -\kappa v_t, v_t\rangle+tC|v_t|^2\geqslant\langle -Av_t, v_t\rangle+tC|v_t|^2\geqslant 0,\quad \forall p\in D^+u(x),
$$
where $\kappa>0$ is the smallest eigenvalue of $A$. Therefore,
$$
(tC-\kappa)|v_t|^2\geqslant 0,\quad \forall p\in D^+u(x).
$$
If $t\leqslant \kappa C^{-1}$, then $v_t\equiv0$, $y_t\equiv x$ and $u(x)=\breve{T}_tu(x)$. Conversely, if $0=D\breve{T}_tu(x)$, then $v_t=0$ and $y_t\equiv x$. It follows $0\in D^+u(x)$ which proves (P5).
\end{proof}

\subsection{What happens for the inf-convolution}
In this section, we will discuss the procedure of inf-convolution, that is, let $u$ be a locally semiconcave function on $\R^n$, and let $L$ be a $C^2$ Tonelli Lagrangian, for any fixed $x\in\R^n$, define
$$
\phi_t(y):=u(y)+A_t(y,x),\quad y\in\R^n.
$$
It is worth noting that $\phi_t$ is the sum of two locally semiconcave functions, and it is also locally semiconcave consequently.

For the convenience of our discussion, we suppose that $u$ is a global viscosity solution of the Hamilton-Jacobi equation
\begin{equation}\label{eq:HJ_1}
	H(x,Du(x))=0,\quad x\in\R^n,
\end{equation}
where $H$ is the associated Hamiltonian with respect to $L$. 

At this stage, we have
$$
u(x)=T_tu(x)=\inf_{y\in\R^n}\phi_t(y)
$$
for all $t>0$ by well known facts from weak KAM theory. 

\begin{Lem}
Let $u$ be a viscosity solution of \eqref{eq:HJ_1}, and $\phi_t$ is defined as above for $t>0$. Then for $t>0$, there exists $z_t$ such that
$$
\phi_t(z_t)=\inf_{y\in\R^n}\phi_t(y).
$$
\end{Lem}

\begin{proof}
	This is actually obvious. Indeed, by Proposition \ref{reachable_grad_and_backward}, for any $t>0$, and $p\in D^*u(x)$, there exists a $C^2$ curve $\gamma:(-\infty,t]\to\R^n$ such that $\gamma(t)=x$, $p=L_v(\gamma(t),\dot{\gamma}(t))$ and
	$$
	u(\gamma(t))-u(\gamma(s))=\int^t_sL(\gamma(\tau),\dot{\gamma}(\tau))\ d\tau,\quad\forall s<t.
	$$
	Take $z_t=\gamma(0)$, then we have the expected result.
\end{proof}

Now, we can impose such a question: Is the aforementioned procedure of inf-convolution efficient for tracking the information of the propagation of singularities along generalized characteristics? 

We will try to answer this question using the technique from nonsmooth critical point theory , see also \cite{Cannarsa-Cheng2} for the applications by standard using Lasry-Lions regularization.	

\begin{Lem}\label{one_to_one}
	Let $u$ be a viscosity solution of \eqref{eq:HJ_1} and the function $\phi_t$ be defined as above for any fixed $x\in\R^n$ and $t>0$. Then there exists a one-to-one correspondence between $p\in D^*u(x)$ and the global minimizers $z_t$ of $\phi_t$ for all $t>0$.
\end{Lem}

\begin{proof}
	Let $z_t\in\R^n$ be a minimizer of $\phi_t$, $t>0$, then $\phi_t$ is differentiable at $z_t$ since $\phi_t$ is locally semiconcave. Thus, $z_t$ is a differentiable point for both $u$ and $A_t(\cdot,x)$. Consequently, there exists two $C^1$ curves $\gamma_1:(-\infty,0]\to\R^n$ and $\gamma_2\in\Gamma^t_{z_t,x}$ such that
	\begin{gather*}
					\gamma_1(0)=\gamma_2(0)=z_t,\\
					p=Du(z_t)=L_v(\gamma_1(0),\dot{\gamma}_1(0)),\\
					p'=D_xA_t(z_t,x)=-L_v(\gamma_2(0),\dot{\gamma}_2(0)),
	\end{gather*}
	by Proposition \ref{reachable_grad_and_backward} and Proposition \ref{C11_A_t}, and $p+p'=0$ since $z_t$ is a critical point of $\phi_t$. Moreover, $\gamma_1$ is a $(u,L,0)$-calibrated curve, i.e., for any $s>0$,
	$$
	u(\gamma_1(0))-u(\gamma_1(-s))=\int^0_{-s}L(\gamma_1(\tau),\dot{\gamma}_1(\tau))\ d\tau,
	$$
	and, similarly,
	$$
	u(x)-u(\gamma_2(0))=A_t(\gamma_2(0),x)=\int^t_0L(\gamma_2(\tau),\dot{\gamma}_2(\tau))\ d\tau.
	$$
	By the juxtaposition of $\gamma_1$ and $\gamma_2$, we define
	$$
	\eta_t(\tau)=\left\{
       \begin{array}{ll}
         \gamma_1(\tau), & \hbox{$\tau\leqslant 0$;} \\
         \gamma_2(\tau), & \hbox{$0<\tau\leqslant t$.}
       \end{array}
     \right.
	$$
	It is clear that $\eta_t$ is a $C^1$ curve on $(-\infty,t]$ with $\eta_t(t)=x$, and
	$$
		u(x)-u(\eta_t(-s))=\int^t_{-s}L(\eta_t(\tau),\dot{\eta}_t(\tau))\ d\tau,\quad s>0
	$$
	which follows that $\eta_t$ is also a $(u,L,0)$-calibrated curve, and such a $(u,L,0)$-calibrated curve passing through $z_t$ with $x$ the terminal datum is unique. Therefore, the correspondence between $z_t$ and $\eta_t$ is one-to-one.

	The rest of the proof is a direct consequence of Proposition \ref{reachable_grad_and_backward}.
\end{proof}

Now, we fix a point $x\in\R^n$. 
\begin{enumerate}[(1)]
  \item If $x$ is a differentiable (or regular) point of $u$, then $D^*u(x)=\{Du(x)\}$, and $\phi_t$ has a unique global minimizer $z_t$ which determines a unique $(u,L,0)$-calibrated curve passing though $z_t$ with $x$ the terminal endpoint point.
  \item If $x$ is singular point of $u$, it become relatively complicated. Let
$$
Z_{x,E}=\{p\in\R^n: H(x,p)\leqslant E\},
$$
which is a non-empty compact and convex set when the energy $E$, say $E=0$, is suitably chosen. It is known that $D^*u(x)=\mathrm{Ext}\,D^+u(x)$, the set of extremal points of $D^+u(x)$, by Proposition \ref{Ext_and_reachable}. This means the elements of $D^*u(x)$ is exactly the set $\mathrm{Ext}\,D^+u(x)$ which is located in the energy hypersurface $\partial Z_{x,E}$ since $H(x,\cdot)$ is strictly convex.
\end{enumerate}

In the spirit of Lemma \ref{one_to_one}, we want to look for the critical points of $\phi_t$. A point $x\in\R^n$ is a {\em critical point} of a locally semiconcave function $u$ if $0\in D^+u(x)$. To find the critical points of $\phi_t$ besides the global minimizers as in Lemma \ref{one_to_one}, we can not apply the standard Lasry-Lions regularization directly since such a function $\phi_t$ is only locally semiconcave. Fortunately, recall a well known nonsmooth critical point theorem, see, for instance, \cite{Shi}. We only need the result in the following finite dimension setting.

\begin{Pro}\label{nonsmooth_MPT}
Let $f:\R^n\to\R$ be a locally Lipschitz function. Suppose that $x_1,x_2\in\R^n$, $x_2\not\in\bar{B}(x_1,r)$ with $r>0$ such that	
$$
\max\{f(x_1),f(x_2)\}<b_0<\inf_{\partial B(x_1,r)}f,
$$
and define
$$
b=\inf_{\gamma\in\Gamma}\max_{t\in[0,1]}f(\gamma(t)),
$$
where $\Gamma=\{\gamma\in C([0,1],\R^n): \gamma(0)=x_1, \gamma(1)=x_2\}$. If $f$ is coercive, then there exists $x_3$ such that $f(x_3)=b$ and $0\in\partial f(x_3)$, where $\partial f(x_3)$ is the Clarke's generalized gradient of $f$ at $x_3$.
\end{Pro}

The readers can refer to \cite{Clarke} for the definition and properties of Clarke's generalized gradients. Applying Proposition \ref{nonsmooth_MPT} to $f=\phi_t$ above, we obtain

\begin{Lem}\label{nonsmooth_crtical_point}
Let $u$ be a Lipschitz viscosity solution of \eqref{eq:HJ_1}, $t>0$, and let $x\in\R^n$ be a singular point of $u$. Suppose there exist finite many elements in $D^*u(x)$, say $D^*u(x)=\{p_1,\ldots,p_k\}$ with $k\geqslant2$, then there exist $k$ distinct global minimizers $z^1_t,\ldots,z^k_t$ of $\phi_t$. 

Moreover, if 
\begin{equation}\label{eq:critical_value}
	b_{ij}=\inf_{\gamma_{ij}\in\Gamma_{ij}}\max_{s\in[0,1]}\phi_t(\gamma_{ij}(s)),\quad 1\leqslant i,j\leqslant k,\ i\not= j,
\end{equation}
where $\Gamma_{ij}=\{\gamma\in C([0,1],\R^n): \gamma(0)=z^i_t, \gamma(1)=z^j_t\}$, then, for each pair of $(i,j)$ with $i\not=j$, there exists a third critical point $x_t^{ij}$ of $\phi_t$ such that $\phi_t(x_t^{ij})=b_{ij}>\min_{y\in\R^n}\phi_t(y)$.
\end{Lem}

\begin{proof}
	We suppose $k=2$ and the proof in the general case is definitely similar. Suppose $\{p_1,p_2\}=D^*u(x)$, and $t>0$, then there exists two $(u,L,0)$-calibrated $C^1$ curves $\eta^1_t$ and $\eta^2_t$, and two global minimizers $z^1_t$ and $z^2_t$ of $\phi_t$ such that
$$
\eta^1_t(t)=\eta^2_t(t)=x,\quad \eta^1_t(0)=z^1_t,\ \eta^2_t(0)=z^2_t,
$$
by Lemma \ref{one_to_one}. Since $z^1_t,z^2_t$ are two isolated (global) minimizers of local Lipschitz function $\phi_t$ which is coercive by the superlinear growth condition on $L$, then a third critical point is obtained in the context of mountain pass method as in Proposition \ref{nonsmooth_MPT}. The rest part of the proof is a direct consequence of Proposition \ref{nonsmooth_MPT} and the facts the Clarke's generalized gradient $\partial\phi_t(\cdot)$ coincides with the (proximal) superdifferential $D^+\phi_t(\cdot)$ (see \cite{Cannarsa-Sinestrari}) since $\phi_t$ is locally semiconcave.
\end{proof}

Therefore, we claim that if $x\in\R^n$ is a singular point of viscosity solution $u$, there exists a third critical point $x_t=x^{12}_t$ of $\phi_t$ determined by two isolated global minimizers $z^1_t$ and $z^2_t$ and Lemma \ref{nonsmooth_crtical_point}. 

\begin{The}
	Let $u$ be a Lipschitz viscosity solution of \eqref{eq:HJ_1}, $t>0$, and let $x\in\R^n$ be a singular point of $u$. Suppose there exist finite many elements in $D^*u(x)$, say $D^*u(x)=\{p_1,\ldots,p_k\}$ with $k\geqslant2$, then there exist critical points $\{x^{ij}_t\}$ of $\phi_t$ (not global minimizers) such that, for $1\leqslant i,j\leqslant k$, $i\not= j$, each critical point $x_t=x^{ij}$ has the following dichotomy: 
\begin{enumerate}[(a)]
  \item $x_t$ is a differentiable point of $\phi_t$ and there exists a local minimal curve connecting $x_t$ and $x$. More precisely, there exists a $C^1$ curve $\gamma:(-\infty,t]\to\R^n$ such that $\gamma(0)=x_t$, $\gamma(t)=x$ and the restriction of $\gamma$ on $(-\infty,0]$ is a $(u,L,0)$-calibrated curve, but $\gamma$ is not a $(u,L,0)$-calibrated curve on $(-\infty,t]$;
  \item $x_t$ is a singular point of $u$.
\end{enumerate}
\end{The}

\begin{proof}
	The existence of such critical points $\{x^{ij}_t\}_{1\leqslant i,j\leqslant k}$ of $\phi_t$ is a direct consequence of Lemma \ref{nonsmooth_crtical_point}. The critical points $\{x^{ij}_t\}$ are not global minimizers of $\phi_t$ since each global minimizer $z^i_t$, $1\leqslant i\leqslant k$, is isolated.
	
	Suppose $x_t$ is not a singular point of $\phi_t$, thus both $u(\cdot)$ and $A_t(\cdot,x)$ is differentiable at $x_t$ and
	$$
	Du(x_t)+D_xA_t(x_t,x)=0
	$$
	since $x_t$ is a critical point of $\phi_t$. Let $p=Du(x_t)$ and $p'=-D_xA_t(x_t,x)$, then there exist two $C^1$ curves $\gamma_1:(-\infty,0]\to\R^n$ and $\gamma_2:[0,t]\to\R^n$ such that
	$$
	\gamma_1(0)=\gamma_2(0)=x_t,\quad\gamma_2(t)=x,
	$$
	and
	$$
	p=L_v(\gamma_1(0),\dot{\gamma}_1(0))=-p'=L_v(\gamma_2(0),\dot{\gamma}_2(0)).
	$$
	It follows that $\dot{\gamma}_1(0)=\dot{\gamma}_2(0)$ and $\gamma$, the juxtaposition of $\gamma_1$ and $\gamma_2$, is a $C^1$ curve which is an extremal. But, $\gamma:(-\infty,t]\to\R$ is not a $(u,L,0)$-calibrated curve, otherwise, $\gamma(0)=x_t$ is a global minimizer of $\phi_t$ by Lemma \ref{one_to_one}.
\end{proof}

\begin{Rem}
It is not clear, even when $t>0$ sufficient small, if the critical point $x^{ij}_t$ of $\phi_t$ found by Lemma \ref{nonsmooth_crtical_point} is close to the singular point $x$ of $u$. It is not hard to prove that when the positive time $t$ tends to $0$, the global minimizers $z^i_t$ and $z^j_t$ tend to $x$ along the direction determined by the associated limiting differentials in $D^*u(x)$ respectively. We hope to dig out more information from this approach in the future.
\end{Rem}

\appendix

\section{Regularity properties of fundamental solutions}\label{app_sec_regularity}
For the details of the proofs of the results in this appendix, the readers can refer to \cite{Cannarsa-Chen-Cheng} and \cite{Cannarsa-Cheng3}, or \cite{Bernard2012} under certain special conditions. 

%

\begin{Pro}\label{Main_bound_Lem}
Let $0<t\leqslant1$, $R>0$ and  suppose $L$ satisfies condition (L1) and (L2). Let $\xi\in\Gamma^t_{x,y}$ be a minimizer for $A_t(x,y)$, $x\in\R^n$, $y\in\bar{B}(x,R)$, and let $p(s)$ be the dual arc of $\xi(s)$. Then we have
\begin{align*}
	\sup_{s\in[0,t]}|\dot{\xi}(s)|\leqslant \Delta(x,R/t),\quad\sup_{s\in[0,t]}|p(s)|\leqslant \Delta(x,R/t),\quad\sup_{s\in[0,t]}|\xi(s)|\leqslant \Delta(x,R/t),
\end{align*}
where $\Delta(x,\cdot)$ is non-decreasing and continuous.
\end{Pro}

\begin{proof}
	For any $t>0$, $R>0$, let $x\in\R^n$, $y\in \bar{B}(x,R)$ and $\xi\in\Gamma^{t}_{x,y}$ be a minimizer for $A_t(x,y)$, i.e., $
A_t(x,y)=\int^t_0L(\xi(s),\dot{\xi}(s))\ ds$. Denoting by $\sigma\in\Gamma^t_{x,y}$ the straight line segment defined by $\sigma(s)=x+\frac st(y-x)$, $s\in[0,t]$, then by the Nagumo conditions in (L2) we have
	\begin{align*}
		&\int^t_0\theta(|\dot{\xi}(s)|)\ ds-c_0t\leqslant\int^t_0L(\xi(s),\dot{\xi}(s))\ ds\leqslant\int^t_0L(\sigma(s),\dot{\sigma}(s))\ ds\\
		=&\int^t_0L\left(x+\frac st(y-x),\frac{y-x}t\right)-L\left(\frac st(y-x),\frac{y-x}t\right)+L\left(\frac st(y-x),\frac{y-x}t\right)\ ds\\
		\leqslant&c_1t|x|\theta\left(\left|\frac{y-x}t\right|\right)+t\max_{y\in\bar{B}(x,R),s\in[0,t]}\left|L\left(\frac st(y-x),\frac{y-x}t\right)\right|\\
		\leqslant&c_1t|x|\theta\left(\left|\frac{y-x}t\right|\right)+tM(t,R)\\
		\leqslant&c_1t|x|\theta(R/t)+t\max_{|x|,|v|\leqslant R/t}|L(x,v)|:=C_1(t,R).
	\end{align*}
	By condition (L2), we have
	\begin{align*}
		|L(x,v)-L(0,0)|\leqslant& |L(x,v)-L(x,0)|+|L(x,0)-L(0,0)|\\
		\leqslant& c_1\theta(R/t)|v|+c_2(x).
	\end{align*}
	Thus,
	\begin{equation}\label{eq:Lip_2}
			C_1(t,R)\leqslant c_3(x,R)t\kappa_1(R/t)
	\end{equation}
	with $\kappa_1(s)=\theta(s)(1+s)+1$. By the superlinear growth condition of $\theta$, we have that
	\begin{equation*}
		\int^t_0|\dot{\xi}(s)|\ ds\leqslant c_4(x,R)t\kappa_2(R/t),
	\end{equation*}
	where $\kappa_2(s)=1+\kappa_1(s)$. Hence
	\begin{equation}\label{eq:bound_xi}
		|\xi(s)-x|\leqslant\int^s_0|\dot{\xi}(s)|\ ds\leqslant c_4(x,R)t\kappa_2(R/t),\quad s\in[0,t],
	\end{equation}
	and
	\begin{equation}\label{eq:inf_dot_xi}
			\inf_{s\in[0,t]}|\dot{\xi}(s)|\leqslant\frac1{t}\int^t_0|\dot{\xi}(s)|\ ds\leqslant c_4(x,R)t\kappa_2(R/t).
	\end{equation}
	Now, we turn to estimate $\sup_{s\in[0,t]}|\dot{\xi}(s)|$. By condition (L2), we have
	\begin{equation}\label{eq:sup}
		\begin{split}
			\theta(|\dot{\xi}(s)|)\leqslant&L(\xi(s),\dot{\xi}(s))\leqslant L(\xi(s),0)+\langle L_v(\xi(s),\dot{\xi}(s)),\dot{\xi}(s)\rangle\\
		=&L(\xi(s),0)+\left\langle \int^s_0L_x(\xi(\tau),\dot{\xi}(\tau))\ d\tau+L_v(\xi(0),\dot{\xi}(0)),\dot{\xi}(s)\right\rangle.
		\end{split}
	\end{equation}		
	Note that we use the Euler-Lagrange equation in the last equality. By condition (L2) and the estimates above, we have
	\begin{align*}
		\int^s_0|L_x(\xi(\tau),\dot{\xi}(\tau))|\ d\tau\leqslant\int^s_0c_1\theta(|\dot{\xi}(\tau)|)\ d\tau\leqslant c_4(x,R)t\kappa_2(R/t).
	\end{align*}
	For any $s\in[0,t]$, we also have
	\begin{align*}
			|L_v(\xi(0),\dot{\xi}(0))|\leqslant|L_v(\xi(s),\dot{\xi}(s))|+\int^s_0|L_x(\xi(\tau),\dot{\xi}(\tau))|\ d\tau.
	\end{align*}
	Then, by condition (L2) and \eqref{eq:inf_dot_xi}, it follows
	\begin{align*}
			|L_v(\xi(0),\dot{\xi}(0))|\leqslant c_1\theta(|\dot{\xi}(s)|)+c_4(t,R)t\kappa_2(R/t),
	\end{align*}
	and this implies
	\begin{align*}
		|L_v(\xi(0),\dot{\xi}(0))|\leqslant& c_1\theta(\inf_{s\in[0,t]}|\dot{\xi}(s)|)+c_4(t,R)t\kappa_2(R/t)\\
		\leqslant& c_1\theta(c_3(x,R)t\kappa_1(R/t)+c_0t)+c_4(x,R)t\kappa_2(R/t).
	\end{align*}
	It follows there exists $M>0$ and $\mu>0$ such that
	\begin{align*}
		|\dot{\xi}(s)|\leqslant \frac 1M\{c_1\theta(c_3(x,R)t\kappa_1(R/t)+c_0t)+2c_4(t,R)t\kappa_2(R/t)+\mu\}.
	\end{align*}
	So, if $t\leqslant1$, we have
	\begin{align*}
			\sup_{s\in[0,t]}|\dot{\xi}(s)|\leqslant& C_2(x,R/t)
	\end{align*}
As for the dual arc $p(\cdot)$, by (L2), we have
$$
\sup_{s\in[0,t]}|p(s)|=\sup_{s\in[0,t]}|L_v(\xi(s),\dot{\xi}(s))|\leqslant C_3(x,R/t).
$$
We complete the proof by defining $\Delta(x,R/t)=\max\{C_2(x,R/t),C_3(x,R/t)\}$.
\end{proof}



Fix $x\in\R^n$ and suppose $R>0$ and $L$ satisfies condition (L1)-(L2). In this case, the following observation is one of the key points of the results on the local regularity properties of $A_t(x,y)$.. For any $t>0$ and $y\in\bar{B}(x,R)$, let $\xi_{t,y}\in\Gamma^t_{x,y}$ be a minimizer for $A_t(x,y)$, and $p_{t,y}$ be its dual arc, then we have
$$
\sup_{s\in[0,t]}|\dot{\xi}_{t,y}(s)|\leqslant\Delta(x,R/t),\quad
	\sup_{s\in[0,t]}|p_{t,y}(s)|\leqslant \Delta(x,R/t),
$$
by Proposition \ref{Main_bound_Lem}. Now, define
\begin{equation}\label{eq:K}
	\begin{split}
	\mathbf{K}_{x}&:=\bar{B}(x,\Delta(x,1))\times\bar{B}(0,\Delta(x,1))\subset\R^n\times\R^n,\\
	\mathbf{K}^*_{x}&:=\bar{B}(x,\Delta(x,1))\times\bar{B}(0,\Delta(x,1))\subset\R^n\times(\R^n)^*.
	\end{split}
\end{equation}
Then, by defining a function $R(x,\cdot):\R^n\times(0,1]\to(0,\infty)$, $R(x,t)=\frac t2$, we have
\begin{equation}\label{eq:key}
	\Delta(x,1/2)\leqslant\Delta(x,1).
\end{equation}
because of the monotonicity properties of $\Delta(x,\cdot)$ and the continuity. So, if $y\in\bar{B}(x,R(x,t))$, and $\xi_t\in\Gamma_{x,y}$ is a minimizer in the definition of $A_t(x,y)$, then
$$
\{\xi(s),p(s))\}_{s\in[0,t],t\in(0,1]}\subset \mathbf{K}^*_{x},\quad \{\xi(s),\dot{\xi}(s))\}_{s\in[0,t],t\in(0,1]}\subset \mathbf{K}^*_{x},
$$

\begin{Pro}[\cite{Cannarsa-Cheng3}]\label{equi_Lip}
Fix any $x\in\R^n$ and $t>0$ with $R(x,t)$ defined as in \eqref{eq:key}. If $y_t$ is the unique maximizer of $\psi_t$ in $\bar{B}(x,R(x,t))$ for all $t\in(0,t_0]$, and $\xi_t\in\Gamma^t_{x,y_t}$ is a minimal curve in the definition of $A_t(x,y_t)$, $t\in(0,t_0]$, then the family $\{\dot{\xi}_t\}$ is equi-Lipschitz.
\end{Pro}

The proof of the following result is similar to those in \cite{Cannarsa-Cheng3} since the estimates involving certain first and second order partial derivatives of $L$ or $H$ which are bounded on the {\em a priori} compact sets $\mathbf{K}^*_{x}$ or $\mathbf{K}_{x}$. The difference between the cases here and what in \cite{Cannarsa-Cheng3} is that the bound for the minimal curves and the dual arc is independent of $x$ in the latter.

\begin{Pro}\label{convexity_A_t}
Suppose $L$ is a Tonelli Lagrangian satisfying (L1)-(L2). Fix any $x\in\R^n$, then there exists $t_0>0$, such that for $0<t\leqslant t_0$, $(t,y)\mapsto A_t(x,y)$ is locally convex in 
$$
S(x,t_0)=\{(t,y)\in\R\times\R^n: 0<t\leqslant t_0, |y-x|\leqslant R(x,t)\},
$$
with $R(x,t)$ defined in \eqref{eq:key}. 

More precisely, there exists constants $C_1,C_2>0$ such that, if $y\in B(x,R(x,t))$, then, for $|h|\ll 1$ and $|z|\ll1$, we have
\begin{equation}\label{eq:convextity_local}
A_{t+h}(x,y+z)+A_{t-h}(x,y-z)-2A_t(x,y)\geqslant \frac{C_1}{t^3}|h|^2+\frac {C_2}t|z|^2.
\end{equation}

\end{Pro}

\begin{Pro}\label{C11_A_t}
Suppose $L$ is a Tonelli Lagrangian satisfying (L1)-(L2). For any $x\in\R^n$, there exists $t_0>0$, such that the functions $w:(t,y)\mapsto A_t(x,y)$ and $(t,y)\mapsto A_t(y,x)$ are both of class $C^{1,1}_{\text{loc}}$ in
$$
S(x,t_0)=\{(t,y)\in\R\times\R^n: 0<t\leqslant t_0, |y-x|\leqslant R(x,t)\},
$$
with $R(x,t)$ defined in \eqref{eq:key}, for $0<t\leqslant t_0$. In Particular, for any $t\in(0,t_0]$,
\begin{align}
D_yA_t(x,y)=&L_v(\xi(t),\dot{\xi}(t)),\label{eq:diff_A_t_y}\\
D_xA_t(x,y)=&-L_v(\xi(0),\dot{\xi}(0)),\label{eq:diff_A_t_x}\\
D_tA_t(x,y)=&-E_{t,x,y},\label{eq:diff_A_t_t}
\end{align}
where $\xi\in\Gamma^t_{x,y}$ is the unique minimizer for $A_t(x,y)$ and $E_{t,x,y}$ is the energy of the Hamiltonian trajectory $(\xi(s),p(s))$ with $p(s)=L_v(\xi(s),\dot{\xi}(s))$.
\end{Pro}

\end{document}